\documentclass{svmult}
\usepackage[utf8]{inputenc}
\usepackage{tikz}
\usepackage{newtxtext}
\usepackage{newtxmath}
\usepackage{makeidx}
\usepackage{multicol}
\usepackage{footmisc}
\usepackage{graphicx}
\usepackage{booktabs}
\usepackage{multirow}
\usepackage[justification=centering]{caption}
\usepackage{tabularx}
\usepackage{comment}
\usepackage{pgfplots, bm,tabu}
\usepackage{array}
\usepackage{cite}
\usepackage{bbm}
\usepackage[linesnumbered,lined,boxed,commentsnumbered]{algorithm2e}

\usetikzlibrary{positioning, calc}
\usetikzlibrary{decorations.markings}

\DeclareMathOperator*{\argmax}{arg\,max}

\newcommand{\R}{\mathbb{R}}

\newcommand{\ang}[1]{\langle #1 \rangle}

\newcommand{\ra}[1]{\renewcommand{\arraystretch}{#1}}

\begin{document}
\title*{On Low-Rank Convex-Convex Quadratic Fractional Programming}
\author{Ilya Krishtal \and  Brendan Miller}
\institute{Ilya Krishtal  
\and Brendan Miller 
\at Northern Illinois University, Dekalb, IL 60435 USA \\ \email{ikrishtal@niu.edu (IK), bmiller14@niu.edu (BM)} }

\maketitle
\abstract{
    We present an efficient algorithm for solving fractional programming problems whose objective functions are the ratio of a low-rank quadratic to a positive definite quadratic with convex constraints. The proposed algorithm for these convex-convex problems is based on the Shen-Yu Quadratic Transform \cite{shen_yu_2018} which finds stationary points of concave-convex sum-of-ratios problems. We further use elements of the algorithm proposed in \cite{shen_yu_2018} and the classic Dinkelbach approach to ensure convergence. We show that our algorithm performs better than previous algorithms for low-rank problems.    
}
\section{Introduction}
Methods of fractional programming encompass a large range of techniques to solve problems of the form 
\begin{equation}\label{sumofratiosproblem}
\begin{aligned}
\max &\quad F(x) = \sum_{m = 1}^M \frac{N_m(x)}{D_m(x)} \\
\text{s.t.} &\quad x \in X \subset \R^n
\end{aligned}
\end{equation}
where $N_m(x), D_m(x): \R^n \to \R$ are continuous functions, and $X$ is a closed, convex set. A convention of the field, which we adopt throughout this paper, is that $N_m(x) \geq 0$ and $D_m(x) > 0$ for all $x \in X$. The problem is called \emph{single-ratio} if $M = 1$ and a \emph{sum-of-ratios} fractional programming problem if $M > 1$. Fractional programming problems arise in many different applications, such as finance and portfolio analysis, government contracting, and engineering (see \cite{benson_2002,benson1_2002,lo_mackinlay_1997,jiao_liu_2017}). We were led to a single-ratio problem in the process of studying the problem of optimal sensor placement for dynamical sampling on graphs \cite{HNT22}. In particular, the relative error of reconstructing a signal on a graph from spatio-temporal samples in the presence of noise can be bounded above by a ratio of two quadratics depending on the eigenvalues of a certain frame operator.

Fractional programs have a long history. Below we outline a few milestones in the development of the theory.

In 1962, Charnes and Cooper showed in \cite{charnes_cooper_1962} that linear fractional programming problems, which are single-ratio fractional programming problems whose numerator, denominator, and constraints are linear, can be solved by simplex method. 

In 1967, Dinkelbach \cite{dinkelbach_1967} showed that any concave-convex single-ratio fractional programming problem can be solved efficiently by consecutively solving several concave maximization problems. The approach used by Dinkelbach was so useful that it has become standard and even seen many applications to fractional programs which are not concave-convex. In the latter case, however, the method involves solving a succession of non-concave maximization problems and becomes rather expensive.

In 1997, Lo and MacKinlay \cite{lo_mackinlay_1997} proposed the problem of maximizing the ratio of two convex quadratic functions in the context of portfolio analysis. The following decade saw two papers published \cite{gotoh_konno_2001, yamamoto_konno_2007} which presented algorithms for solving this problem exactly. Both follow the standard Dinkelbach approach, with different methods of handling the most expensive aspect of the problem: the need to solve multiple non-convex quadratic programming problems. To do so, the former paper by Gotoh and Konnoh implemented a branch-and-bound technique while the latter by Yamamoto and Konnoh iteratively approximated the quadratic by a piece-wise linear function, which is maximized by standard mixed integer linear programming techniques.

In recent years, there has also been much research on the sum-of-ratios problem. For example, \cite{benson_2002,benson1_2002,gao_mishra_shi_2010} all produce branch-and-bound type algorithms for solving the concave-convex sum-of-ratios problem. Similarly, \cite{qu_zhang_zhao_2007,jiao_liu_2017} propose efficient algorithms for solving quadratically constrained quadratic sum-of-ratios type problems. Of particular interest to us is the Shen-Yu Quadratic Transform introduced in \cite{shen_yu_2018} which, like the Dinkelbach method, iteratively solves a concave programming problem to converge to a stationary point of the concave-convex sum-of-ratios problem.  

In this paper, we propose an algorithm for solving the single-ratio convex-convex quadratic programming problem which can be effectively utilized when the numerator has low rank. We divide the feasible region into several subregions in which the Shen-Yu Quadratic Transform method can be applied successively until global convergence. Since the methods of Shen and Yu only guarantee convergence to a stationary point, we use the mixed integer linear programming techniques employed in \cite{yamamoto_konno_2007} to check if a given stationary point is a global optimum. Although, as we will see empirically, this rather expensive procedure can often be omitted altogether (with a minimal chance of error). This yields a very efficient algorithm which, with high probability, converges to the global solution of a convex-convex quadratic programming problem by successively solving concave maximization problems. The only known algorithms for solving such problems involve successively maximizing a nonconvex quadratic programming problems. 

The remainder of this paper is organized as follows. In Section \ref{basic}, we state the problem we focus on and outline the classic Dinkelbach approach for ratio maximization. We then convert our problem into a sum-of-ratios problem and recall a recent Shen-Yu scheme for solving such problems. Section \ref{mains} is the centerpiece of the paper. In it, we present a natural way of subdividing the feasible region of our problem into a number of subregions thereby replacing a convex-convex sum-of-ratios problem with a finite number of much simpler concave-convex sum-of-ratios problems. This results in a region checking algorithm (Algorithm \ref{our_algo}) which encompasses our approach to solving single-ratio quadratic convex-convex problems. Numerical experiments described in Section \ref{num} illustrate the effectiveness of our approach in comparison with other algorithms. Finally, concluding remarks are presented in Section \ref{conc}.

\section{Basic Analysis}\label{basic}
We study the problem of the form

\begin{equation}\label{original_problem}
\begin{aligned}
\max &\quad F(x) = \frac{x^T Q x}{x^T P x} \\
\text{s.t.} &\quad Ax \leq b
\end{aligned}
\end{equation}
where $Q$ is an $n \times n$ positive semidefinite matrix, $P$ is an $n \times n$ positive definite matrix, $A \in \R^{T \times n}$, $b \in \R^T$. We will denote the feasible region of this problem by $X \subset \R^n$. Such optimization problems are nonconcave in general, and thus require expensive algorithms to solve. We recall the standard Dinkelbach method which utilizes the function 
\[
\pi(\lambda) := \max_{x \in X} \; x^T(Q - \lambda P)x, \;\; \lambda > 0. 
\]
It is convenient to introduce the following notation:
\[
x(\lambda) = \argmax_{x \in X} \; x^T(Q - \lambda P)x, \;\; \lambda > 0.
\]

\begin{theorem}
The function $\pi(\lambda)$ is convex and strictly decreasing in $\lambda$. Furthermore, $\pi(\lambda) = 0$ if and only if $x(\lambda)$ maximizes \emph{(\ref{original_problem})} in $X$. 
\end{theorem}

The algorithms of \cite{gotoh_konno_2001, yamamoto_konno_2007} are root-finding algorithms which use a scheme developed by Ibaraki \cite{ibaraki_1983} to search for the root of $\pi$. These algorithms become expensive because computing $\pi(\lambda)$ is a nonconvex quadratic programming problem, whose difficulty is larger for smaller values of $\lambda$. We will utilize this Theorem of Dinkelbach in our algorithm, but only to check if a local maximum of (\ref{original_problem}) is indeed a global maximum. We summarize the Ibaraki scheme here as Algorithm \ref{ibaraki} for reference, but omit the explanation of convergence.

\IncMargin{1em}
\begin{algorithm}
\SetKwData{Left}{left}\SetKwData{This}{this}\SetKwData{Up}{up}
\SetKwFunction{Union}{Union}\SetKwFunction{FindCompress}{FindCompress}
\SetKwInOut{Input}{input}\SetKwInOut{Output}{output}
\Input{Matrices $Q$ and $P$, linear inequality constraints $Ax \leq b$, and a tolerance $\varepsilon$.}
\BlankLine
Find $\lambda^u$ with $\pi(\lambda^u) < 0$ and $\lambda^l$ with $\pi(\lambda^l) > 0$\;
\Repeat{$\vert \pi(\overline{\lambda})\vert < \varepsilon$}{
Compute $\overline{\lambda}$ as follows:
\begin{equation*}
    \overline{\lambda} = \begin{cases}
     -\frac{\pi(\lambda^u)}{\Delta \pi} + \lambda^u & \text{if } x(\lambda^u)^TPx(\lambda^u) + \Delta \pi \neq 0\\
     \frac{\pi(\lambda^u)}{x(\lambda^u)^TPx(\lambda^u)} + \lambda^u & \text{otherwise,}
    \end{cases}
\end{equation*}
where $\Delta \pi = (\pi(\lambda^u) - \pi(\lambda^l))/(\lambda^u - \lambda^l)$\;
Compute $\pi(\overline{\lambda})$\;
Update $\lambda^l = \overline{\lambda}$ if $\pi(\overline{\lambda}) > 0$, or $\lambda^u = \overline{\lambda}$ if $\pi(\overline{\lambda}) < 0$\;
}
\caption{Interpolated Binary Search (Ibaraki scheme)}\label{ibaraki}
\end{algorithm}\DecMargin{1em}

The algorithms presented in \cite{gotoh_konno_2001,yamamoto_konno_2007} both utilize Algorithm \ref{ibaraki}, but employ different methods of computing $\pi(\overline{\lambda})$. We will use contemporary software to solve these non-convex quadratic programming problems to ensure the most accurate solutions. 

Suppose now that $Q$ can be written as
\[
Q = \sum_{m = 1}^M q_mq_m^T,
\]
which is a sum of rank-one matrices. Note that $Q$ always admits such a decomposition as we can take $M = rank(Q)$ and $q_m = \sqrt{\lambda_m}v_m $ where $v_m$ is the eigenvector of $Q$ associated to the nonzero eigenvalue $\lambda_m$. Then we may rewrite the objective function in the following way:
\begin{equation}\label{rearrange}
\begin{aligned}
F(x) &= \frac{\sum_{m = 1}^M x^T(q_mq_m^T)x}{x^TPx} \\
     &= \sum_{m = 1}^M \frac{\ang{q_m,x}^2}{x^TPx}.
\end{aligned}
\end{equation}
This reformulation of the objective function converts a single-ratio fractional programming problem \eqref{original_problem} into a convex-convex sum-of-ratios fractional programming problem which we will refer to as \eqref{rearrange}. 

Next, we summarize a method for suboptimally solving sum-of-ratios fractional programming problems (\ref{sumofratiosproblem}) that was introduced in \cite{shen_yu_2018}.

\begin{definition}[See \cite{shen_yu_2018}]
Given the sum-of-ratios fractional programming problem \emph{(\ref{sumofratiosproblem})}, the Shen-Yu Quadratic Transform of this problem is defined to be 
\begin{equation}\label{quadtrans}
g(x,y) = \sum_{m = 1}^M \left(2y_m\sqrt{N_m(x)} - y_m^2 D_m(x)\right). 
\end{equation}
\end{definition}

\begin{theorem}[\cite{shen_yu_2018}]\label{equivmax} The sum-of-ratios problem \emph{(\ref{sumofratiosproblem})} and its quadratic transform \emph{(\ref{quadtrans})} maximization problem are equivalent. More precisely, the maximum of \emph{(\ref{quadtrans})} occurs at $(x^*,y^*)$ where $x^*$ maximizes \emph{(\ref{sumofratiosproblem})} and $y^* = (y_m^*)_{m = 1}^M$ satisfies
\[
y_m^* = \frac{\sqrt{N_m(x^*)}}{D_m(x^*)}.
\]
\end{theorem}

The following Lemma is useful for deriving the Shen-Yu scheme.

\begin{lemma}\label{y_lemma}
Let $g(x,y)$ be as in \emph{(\ref{quadtrans})}. If $x_0 \in X$, then $y_0 := \argmax_{y \in \R^M} g(x_0,y)$ can be found analytically and is given by $y_0 = (y_m^0)_{m = 1}^M$, where
\[
y_m^0 = \frac{\sqrt{N_m(x_0)}}{D_m(x_0)}.
\]
\end{lemma}

\begin{proof}
It suffices to maximize each term in the sum in \eqref{quadtrans} individually. Clearly, the vertex of the quadratic $f(y_m) = 2y_m\sqrt{N_m(x_0)} - y_m^2 D(x_0)$ occurs at the point 
\[
y_m^0 = \frac{\sqrt{N_m(x_0)}}{D_m(x_0)}
\]
as desired.
\end{proof}

We may now state and prove a theorem from which, when taken together with Theorem \ref{equivmax}, an algorithm for suboptimally solving (\ref{sumofratiosproblem}) is naturally derived. 

\begin{theorem}\label{shenyuthm}{\emph{(Shen, Yu \cite{shen_yu_2018})}}
Consider the sum-of-ratios problem \emph{(\ref{sumofratiosproblem})} and suppose that $x_0 \in X$. Let $y_0 = (y_m^0)_{m=1}^M$ with $y_m^0 = \sqrt{N_m(x_0)}/D_m(x_0)$. If 
\[
x^* = \argmax_{x \in X} g(x,y_0) = \argmax_{x \in X} \sum_{m = 1}^M\left( 2y_m^0\sqrt{N_m(x)} - (y_m^0)^2 D_m(x)\right),
\]
then $F(x^*) \geq F(x_0)$. 
\end{theorem}

\begin{proof}
Let $x_0,y_0,$ and $x^*$ be as above and set 
\[
y_m^* = \frac{\sqrt{N_m(x^*)}}{D_m(x^*)}
\]
and $y^* = (y_m^*)_{m = 1}^M$. Then we have the following string of inequalities
\[
\begin{aligned}
F(x_0) &= g(x_0,y_0) \\
       &\leq g(x^*,y_0) \\
       &\leq g(x^*,y^*) \\
       &= F(x^*)
\end{aligned}
\]
where the third line follows from Lemma \ref{y_lemma}, and the last by a direct computation. The proof is complete.
\end{proof}

The above Theorem guarantees that replacing $x_0$ by $x^*$ and $y_0$ by $y^*$ improves the value of the objective function with each iteration, and thus this method (with the scheme written explicitly in Algorithm \ref{quad_trans_algo}) converges. It is clear that the algorithm  
converges to a local maximum, say $(x^*,y^*)$, of the Quadratic Transform $g$. It follows easily that the value $x^*$ is indeed a local maximum of the objective $F$ as well. 

\IncMargin{1em}
\begin{algorithm}
\SetKwData{Left}{left}\SetKwData{This}{this}\SetKwData{Up}{up}
\SetKwFunction{Union}{Union}\SetKwFunction{FindCompress}{FindCompress}
\SetKwInOut{Input}{input}\SetKwInOut{Output}{output}
\Input{Functions $N_m$ and $D_m$, and a compact, convex set $X$.}
\BlankLine
Find an initial $x_0 \in X$\;
\Repeat{\emph{convergence}}{
set $y_m^0 = \sqrt{N_m(x_0)}/D_m(x_0)$\;
solve $x^* = \argmax_{x \in X}{\sum_{m = 1}^M \left(2y_m^0\sqrt{N_m(x)} - (y_m^0)^2 D_m(x)\right)}$\;
update $x_0 = x^*$\;
}
\caption{Shen-Yu Iterative Algorithm}\label{quad_trans_algo}
\end{algorithm}\DecMargin{1em}

\begin{corollary}\label{shenyucor}
Algorithm \emph{\ref{quad_trans_algo}} converges to a stationary point of the sum-of-ratios fractional programming problem \emph{(\ref{sumofratiosproblem})}.
\end{corollary}

Although Algorithm \ref{quad_trans_algo} is guaranteed to converge to a stationary point, its usefulness is limited to the difficulty of maximizing the Quadratic Transform over the variable $x$. In the case when each $N_m$ is concave and each $D_m$ is convex, this can be done by any method of concave programming. When the objective has the form (\ref{rearrange}), the Quadratic Transform becomes 
\begin{equation}\label{quad_trans_rearrange}
g(x,y) = \sum_{m = 1}^M 2y_m\vert \ang{q_m,x} \vert - y_m^2 x^TPx
\end{equation}
which is not, in general, concave in $x$. However, in the event that the absolute values around each linear term in the sum can be dropped (i.e. each $\ang{q_m,x}$ is either nonnegative or nonpositve valued on $X$), then we may apply Algorithm \ref{quad_trans_algo} effectively. 

\section{Region checking algorithm}\label{mains}

Henceforth we will assume the objective function $F$ is as in (\ref{original_problem}) and, therefore, it can be rewritten in the sum-of-ratios form (\ref{rearrange}). The first observation to make is that combining Algorithm \ref{quad_trans_algo} with the Dinkelbach method yields an algorithm which converges to the global maximum of (\ref{rearrange}). Indeed, if $x^*$ is the local optimum found by Algorithm \ref{quad_trans_algo}, we may set $\lambda^* = F(x^*)$ and compute both $\pi(\lambda^*)$ and $x(\lambda^*)$. If $\vert \pi(\lambda^*) \vert > \varepsilon$ where $\varepsilon$ is some tolerance, we again run Algorithm \ref{quad_trans_algo} with $x(\lambda^*)$ as the initial feasible point and repeat. This method is guaranteed to converge to the globally optimal solution since 
\[
x(\lambda^*)^T(Q - \lambda^* P)x(\lambda^*) > x^{*T}(Q - \lambda^* P)x^*
\]
implies that 
\[
F(x(\lambda^*)) > F(x^*) \geq F(x_0)
\]
where $x_0$ is the initial feasible point used in Algorithm \ref{quad_trans_algo}.

Finding the value of $\pi(\lambda^*)$ and the vector $x(\lambda^*)$ can be achieved by the methods introduced in \cite{gotoh_konno_2001} and refined in \cite{yamamoto_konno_2007}, but maximizing the Quadratic Transform as in Algorithm \ref{quad_trans_algo} cannot. To circumvent this issue, we divide the feasible region into at most $2^{rank(Q)}$ subregions and perform this algorithm independently in each subregion. Suppose, as before, that 
\[
Q = \sum_{m = 1}^M q_mq_m^T,
\]
and for each $1 \leq m \leq M$ define
\[
R_m^0 = \{x \in X \mid \ang{q_m,x} \leq 0\},\;\; R_m^1 = \{x \in X \mid \ang{q_m,x} \geq 0\}.
\]
For each binary sequence $(n_m)_{m = 1}^M = n \in \{0,1\}^M$, we denote by $R^n$ a (possibly empty) subregion of $X$ given by 
\[
R^n = \bigcap_{m = 1}^M R^{n_m}_m. 
\]
A straightforward observation then yields the following result.

\begin{lemma}
The Quadratic Transform \emph{(\ref{quad_trans_rearrange})} of \emph{(\ref{rearrange})} is concave in the $x$ variable over each nonempty subregion $R^n$ for $n \in \{0,1\}^M$. 
\end{lemma}

\begin{proof}

    Let $g(x,y)$ be as in (\ref{quad_trans_rearrange}) and fix a binary sequence $n \in \{0,1\}^M$. Then, by definition of $R^n$, we see that for each $1\leq m \leq M$ we have either $\ang{q_m,x} \geq 0$ or $\ang{q_m,x} \leq 0$ for all $x \in R^n$. This implies that $|\ang{q_m,x}|$ is linear over $R^n$, and hence $g(x,y)$ is a sum of concave functions, which is itself concave. 
\end{proof}

We can now formulate our method for solving the convex-convex quadratic fractional programming problem as Algorithm \ref{our_algo}.

\IncMargin{1em}
\begin{algorithm}
\SetKwData{Left}{left}\SetKwData{This}{this}\SetKwData{Up}{up}
\SetKwFunction{Union}{Union}\SetKwFunction{FindCompress}{FindCompress}
\SetKwInOut{Input}{input}\SetKwInOut{Output}{output}
\Input{Matrices $Q = \sum_{m = 1}^M q_mq_m^T, P$, linear constraints $Ax \leq b$, and a tolerance $\varepsilon$.}
\Output{Global Solution of (\ref{original_problem})}
\BlankLine
\For{$n \in \{0,1\}^M$}{
\If{$R^n \neq \emptyset$}{
Find an initial $x_n \in R^n$\;
\Repeat{\emph{$\vert \pi(\lambda)\vert < \varepsilon$}}{
Find stationary point $x^*$ of (\ref{original_problem}) via Algorithm \ref{quad_trans_algo} with initial point $x_n$\;
Set $\lambda = F(x^*)$\;
Compute $\pi(\lambda)$ and $x(\lambda)$\;
Update $x_n = x(\lambda)$\;
}
}
}
Determine $\max{F(x_n)}$\;
\caption{Region-Checking Algorithm for Convex-Convex Quadratic Fractional Programming
}
\label{our_algo}
\end{algorithm}\DecMargin{1em}

\begin{remark}\label{tpm}
The utility of Algorithm \ref{our_algo} is, in full generality, limited by the rank of $Q$ or, more precisely, by the number of nonempty subregions (which is controlled by the rank of $Q$). It is also worth noting that different decompositions of $Q$ may yield different numbers of nonempty subregions of the feasible region. We leave the question of how to find better decompositions of $Q$ beyond the scope of this paper. We do, however, mention explicitly the case when $Q$ is a totally nonnegative matrix and the optimization problem in question is
\begin{equation*}
\begin{aligned}
\max &\quad F(x) = \frac{x^T Q x}{x^T P x} \\
\text{s.t.} &\quad Ax = b\\
            &\quad 0 \leq x \leq \alpha.
\end{aligned}
\end{equation*}

In this case, we write $Q = LDL^T$ for $D$ a diagonal matrix and $L$ a lower triangular matrix. This gives a decomposition of $Q$ as
\[
Q = \sum_{m = 1}^M d_m \ell_m \ell_m^T,
\]
where $\ell_m$ and $d_m$ are, respectively, the $m^{th}$ column of $L$ and $m^{th}$ diagonal entry of $D$.
Since $Q$ is totally nonnegative, the entries of each vector $\ell_m$ are nonnegative \cite{goodearl_lenagan_2012}, in which case $\ang{\ell_m,x} \geq 0$ for all feasible $x$. Thus, there is only one subregion $R^n$ of the feasible region which is nonempty. 
\end{remark}

It is quite possible that Algorithm \ref{our_algo} is considerably slower than the algorithms proposed in \cite{gotoh_konno_2001} and \cite{yamamoto_konno_2007} given the potentially large number of times $\pi(\lambda)$ is computed. However, as we will show empirically in the next section, often the first local maximum found by Algorithm \ref{our_algo} in a given region is, in fact, the global maximum of the region. Thus, we will also compare the efficiency and accuracy of Algorithm \ref{our_algo} without computing $\pi(\lambda)$ and assuming each $x^*$ found by Algorithm \ref{quad_trans_algo} is a global maximum of the region. This modification of Algorithm \ref{our_algo} is  Algorithm \ref{our_algo2} below.
\IncMargin{1em}
\begin{algorithm}
\SetKwData{Left}{left}\SetKwData{This}{this}\SetKwData{Up}{up}
\SetKwFunction{Union}{Union}\SetKwFunction{FindCompress}{FindCompress}
\SetKwInOut{Input}{input}\SetKwInOut{Output}{output}
\Input{Matrices $Q = \sum_{m = 1}^M q_mq_m^T, P$, linear constraints $Ax \leq b$, and a tolerance $\varepsilon$.}
\Output{Local Solution of (\ref{original_problem})}
\BlankLine
\For{$n \in \{0,1\}^M$}{
\If{$R^n \neq \emptyset$}{
Find an initial $x_n \in R^n$\;
Find stationary point $x^*$ of (\ref{original_problem}) via Algorithm \ref{quad_trans_algo} with initial point $x_n$\;
Update $x_n = x^*$\;
}
}
Determine $\max{F(x_n)}$\;
\caption{Fast Region-Checking Algorithm}\label{our_algo2}
\end{algorithm}\DecMargin{1em}

\begin{remark}
Suppose $Q = qq^T$ is a rank-one matrix. We write the objective function as
\[
F(x) = \frac{\ang{q,x}^2}{x^TPx}.
\]
In this case, there are only two subregions of the feasible region: $R^0$ and $R^1$. Also, we may equivalently maximize the square-root of the objective, which is given by
\[
\sqrt{F(x)} = \begin{cases}
-\ang{q,x}/\sqrt{x^TPx} &x \in R^0\\
\ang{q,x}/\sqrt{x^TPx} &x \in R^1.
\end{cases}
\]
Thus, $\sqrt{F(x)}$ is a concave-convex fractional programming problem in each subregion of the feasible region, and hence can be solved by two applications of Algorithm \ref{ibaraki} where computing $\pi(\lambda)$ is a concave programming problem. This method is superior to Algorithm \ref{our_algo2} as it was shown in \cite{shen_yu_2018} that Algorithm \ref{quad_trans_algo} is slower than the standard Dinkelbach method for standard single-ratio concave-convex fractional programs.
\end{remark}

\section{Numerical Experiments}\label{num}
In this section we conduct several numerical experiments on the following optimization problem:
\[
\begin{aligned}
\max &\quad F(x) = \frac{x^T Q x}{x^T P x} \\
\text{s.t.} &\quad Ax \leq b, \;\; A \in \R^{T\times n}, \;\; b \in \R^T \\
            &\quad \sum_{i = 1}^n x_i = 1 \\
            &\quad 0 \leq x_i \leq 0.1, \;\; i = 1,...,n
\end{aligned}
\]
where $A$ and $b$ have random entries in the interval $[0,1]$ so that the vector $(1/n,...,1/n)$ is feasible, and $Q = X^TX$ and $P = Y^TY$ where $X$ and $Y$ are, respectively, an $M \times n$ and an $n \times n$ random matrix with entries in $[0,1]$. Unless otherwise specified, we will always decompose $Q$ as a sum of rank one matrices according to its eigendecomposition as noted in Section 2. 

We first demonstrate the efficiency and accuracy of Algorithms \ref{our_algo} and \ref{our_algo2} against Algorithm \ref{ibaraki} for various combinations of $(n,M,T)$. Next, we examine the average number of nonempty subregions in Algorithms \ref{our_algo} and \ref{our_algo2}. Finally, we demonstrate the accuracy of Algorithm \ref{our_algo2} in a full-rank example (i.e. $r = n$, which has several local maxima) when there are a small number of subregions (see Remark \ref{tpm}).

All computation was done in MATLAB (on AMD A6-7400K Radeon R5 4.09 GHz processor), using Gurobi 9.5 interface to solve the nonconvex quadratic programming problems involved in Algorithms \ref{ibaraki} and \ref{our_algo}. The tolerance $\varepsilon$ is always set as $\varepsilon = 10^{-3}$. All values in the forthcoming tables are averages of five tests.

\subsection{Algorithm Comparison}
We first give a demonstration of how differing combinations of $(n,M,T)$ affect the computation time of Algorithm \ref{our_algo} in a single nonempty region of the feasible set.
\begin{table*}
\centering
\ra{1.2}
\setlength\tabcolsep{2mm}
\begin{tabular}{@{}rrrrrr|rrrr@{}}\toprule
&  & \multicolumn{4}{c}{CPU Time (sec)} &  \multicolumn{4}{c}{Iterations}  \\ \cmidrule{3-6} \cmidrule{7-10} 
& & $T = 1$ & $T=10$ & $T=30$ & $ T = 50$ & $T=1$ & $T=10$ & $T = 30 $& $T=50$ \\ \midrule
$n = 25$ & $M = 2$ & 0.4657 & 0.3689 & 0.7372 & 0.7063 & 1.2 & 1 & 1 & 1\\
 & 5 & 1.074 & 0.5517 & 1.183 & 0.7725 &  1 & 1.2 & 1 & 1 \\
 & 7 &  1.037 & 2.339 & 1.473 & 1.337 &  1 & 1 & 1 & 1\\
 & 10& 0.5639 & 1.253 & 9.06 & 1.849 & 1.2 & 1 & 1.4 & 1.2
 \\
 ~&~&~&~&~&~&~&~&~&~
 \\
50 & 2 &  2.022 & 2.512 & 2.403 & 5.308 &  1.2 & 1.2 & 1 & 1.2\\
 & 5&3.208 & 6.842 & 15.71 & 19.66 &  1 & 1 & 1.2 & 1\\
 & 7 & 60.85 & 5.055 & 6.284 & 31.49 & 1.4 & 1 & 1.2 & 1\\ 
 ~ & 10 & 8.032 & 126.5 & 663.4 & 49.52 &  1.2 & 1 & 1.4 & 1\\
  ~&~&~&~&~&~&~&~&~&~
 \\
 75 & 2 & 4.261 & 7.612 & 11.04 & 7.077 &  1 & 1 & 1.2 & 1 \\ 
~ & 5 & 5.299 & 32.95 & 58.67 & 79.35 &  1 & 1.2 & 1.4 & 1 \\ 
~ & 7 & 125.9 & 62.65 & 48.03 & 273.3 &  1.4 & 1.2 & 1.4 & 1 \\ 
~ & 10 & 23.8 & 20.38 & 93.34 & 904.7 &  1 & 1 & 1 & 1.4 \\
 \bottomrule
\end{tabular}
\caption{Efficiency of Algorithm 3 for one region}
\label{onereg}
\end{table*}

 Table \ref{onereg} shows that the computation time for Algorithm \ref{our_algo} increases sharply with the number of inequality constraints due to the increasing complexity of solving the nonconvex quadratic subproblems. Likewise, the computation time increases with both the number of variables and the rank of $Q$, albeit not as sharply. Second, none of the problems solved in this experiment took more than two iterations of Algorithm \ref{quad_trans_algo} to converge. In fact, 209 of the 240 problems solved in this experiment converged in just one iteration of Algorithm \ref{quad_trans_algo}. This suggests that the expensive procedure of computing $\pi(\lambda)$ can safely be dropped from the algorithm, if a low probability of missing the exact solution may be tolerated by the application.
 
 In the next two experiments, we fix $T = 10$. We now compare Algorithm \ref{our_algo} with the Algorithm \ref{ibaraki}. We show the results for Algorithm \ref{our_algo} both converging to the global solution and forcing only one iteration per region without computing $\pi(\lambda)$. 
 
\begin{table*}
\centering
\ra{1.2}
\setlength\tabcolsep{2mm}
\begin{tabular}{@{}rrrrr|r@{}}\toprule
& & \multicolumn{3}{c}{CPU Time (sec)} &   \multicolumn{1}{c}{Alg 4 Error}  \\ \cmidrule{3-6}
& & Alg 1 & Alg 3 & Alg 4  & Error ($\%$)   \\ \midrule
$n = 10$ & $M = 2$ &  0.01372 & 0.07733 & 0.09616 & 0 \\ 
        ~ & 3 &  0.006822 & 0.01969 & 0.1347 & 0 \\ 
        ~ & 4 &  0.005291 & 0.02188 & 0.2637 & 0 \\ 
        ~ & 5 &  0.01414 & 0.02462 & 0.5047 & 0 \\ 
        ~ & 7 &  0.01159 & -- & 2.006 & 0 \\ 
        ~ & 10 &  0.01003 & -- & 16.67 & 0 \\ 
        ~ & ~ & ~ & ~ & ~ & ~ \\
        25 & 2 &  0.5673 & 0.9934 & 0.1193 & 0 \\ 
        ~ & 3 &  0.6599 & 1.266 & 0.1932 & 0.08 \\ 
        ~ & 4 &  1.068 & 3.975 & 0.3618 & 0 \\ 
        ~ & 5 &  1.105 & 26.85 & 0.7185 & 0.06 \\ 
        ~ & 7 &  0.7177 & -- & 2.976 & 0.02 \\ 
        ~ & 10 &  0.8667 & -- & 23.35 & 0.02 \\ 
        ~ & ~ & ~ & ~ & ~ & ~ \\
        50 & 2 &  4.932 & 4.649 & 0.1331 & 0.04 \\ 
        ~ & 3 &  5.258 & 11.6 & 0.2104 & 0.05 \\ 
        ~ & 4 &  9.485 & 72.51 & 0.4216 & 0.02 \\ 
        ~ & 5 &  7.618 & 223.9 & 0.91 & 0.01 \\ 
        ~ & 7 &  5.573 & -- & 3.18 & 0.03 \\ 
        ~ & 10 &  8.004 & -- & 26.5 & 0 \\ 
        ~ & ~ & ~ & ~ & ~ & ~ \\
        75 & 2 &  12.63 & 17.12 & 0.1229 & 0.01 \\ 
        ~ & 3 &  14.43 & 273.3 & 0.346 & 0 \\ 
        ~ & 4 &  30.47 & 211.2 & 0.531 & 0.2 \\ 
        ~ & 5 &  45.63 & 868.8 & 1.249 & 0.07 \\ 
        ~ & 7 &  34.33 & -- & 4.407 & 0 \\ 
        ~ & 10 & 53.35 & -- & 35.55 & 0.05 \\
 \bottomrule
\end{tabular}
\caption{Comparison of Algorithms in CPU Time (sec) \\ Error = $(F(x_{opt}) - F(x))/F(x_{opt}) $}
\label{comparison}
\end{table*}
 
 Table \ref{comparison} shows the comparison of Algorithms \ref{ibaraki}, \ref{our_algo}, and \ref{our_algo2}. There are several things to note about these results. First, the error incurred from performing only one iteration per region is negligible; it is always under one percent and quite often is under 0.1 percent. This implies that one need only perform Algorithm \ref{quad_trans_algo} once in each subregion of the feasible region, making irrelevant the need to compute $\pi(\lambda)$. Second, for a given value of $M$, the computation time needed to complete Algorithm \ref{quad_trans_algo} increases steadily as $n$ increases, but does so at a much slower rate for  Algorithm \ref{our_algo2}. 
 
 In fact, Table \ref{large} shows that Algorithm \ref{our_algo2} can be used efficiently for large scale problems when the number of subregions is less than 130. For comparison, the authors in \cite{yamamoto_konno_2007} state that the case when $n = 500$ is within reach via their algorithm only by employing an elaborate local search. 
 \begin{table}[h]
    \centering
    \ra{1.2}
\setlength\tabcolsep{2.2mm}
    \begin{tabular}{rrrr}
    \hline
         &  &  & Algorithm \ref{our_algo2} \\ \hline
        $n = 250$ & $M = 7$ & $T = 10$ & 22.32 \\ 
        500 & 7 & 10 & 117.34 \\ 
        750 & 7 & 10 & 357.42 \\ 
        1000 & 7 & 10 & 743.24 \\ \hline
    \end{tabular}
    \caption{CPU Time (sec) for Algorithm \ref{our_algo2}}
    \label{large}
\end{table}
 
\subsection{Accuracy of Algorithm 4}
The numerical experiments above call into question the accuracy of Algorithm \ref{our_algo2} in high-rank problems with a small number of nonempty regions. Table \ref{onereg} shows that there is typically only one local maximum per subregion, and this may seem to be attributable to the large number of subregions relative to the number of local maxima. Thus, one might conclude that it may be disadvantageous to choose a decomposition of $Q$ with fewer terms as this leads to fewer possible subregions. We aim to show that this is not the case. 

\begin{table*}
\centering
\ra{1.2}
\setlength\tabcolsep{2mm}
\begin{tabular}{@{}rrrrrrr@{}}\toprule
&  & \multicolumn{5}{c}{Number of Subregions}   \\ \cmidrule{3-7} 
& & $T = n/2$ & $T=n$ & $T=3n/2$ & $ T = 2n$ & $T=5n/2$  \\ \midrule
$n = 30$ & $M = 2$ & 2 & 2 & 2 & 1 & 1 \\
 ~ & 3 & 4 & 4 & 3.8 & 2 & 1 \\ 
        ~ & 5 & 16 & 16 & 12 & 3 & 1 \\ 
        ~ & 7 & 64 & 64 & 49.4 & 11 & 1 \\ 
        ~ & ~ & ~ & ~ & ~ & ~ & ~ \\ 
50 & 2 & 2 & 2 & 2 & 1.8 & 1 \\ 
        ~ & 3 & 4 & 4 & 4 & 2.8 & 1 \\ 
        ~ & 5 & 16 & 16 & 14.4 & 5 & 1 \\ 
        ~ & 7 & 64 & 64 & 64 & 19.4 & 1 \\ 
        ~ & ~ & ~ & ~ & ~ & ~ & ~ \\ 
        100 & 2 & 2 & 2 & 2 & 1.4 & 1 \\ 
        ~ & 3 & 4 & 4 & 4 & 2.8 & 1 \\ 
        ~ & 5 & 16 & 16 & 16 & 6.6 & 1 \\ 
        ~ & 7 & 64 & 64 & 64 & 23.4 & 1 \\ 
        ~ & ~ & ~ & ~ & ~ & ~ & ~ \\ 
        150 & 2 & 2 & 2 & 2 & 1.2 & 1 \\ 
        ~ & 3 & 4 & 4 & 4 & 2.8 & 1 \\ 
        ~ & 5 & 16 & 16 & 16 & 8 & 1 \\ 
        ~ & 7 & 64 & 64 & 64 & 13.6 & 1 \\ \hline
 \bottomrule
\end{tabular}
\caption{Average number of subregions checked by Algorithm \ref{our_algo2}}
\label{regions}
\end{table*}

The number of subregions is, of course, determined by the linear constraints. Table \ref{regions} shows that the number of subregions is usually $2^{rank(Q) - 1}$, and this decreases only when the number of inequality constraints is much larger than the number of variables. This implies that the number of subregions checked by the algorithm increases exponentially with the rank of $Q$, and thus the accuracy of Algorithm \ref{our_algo2} could be a product of the brute-force nature of checking each region. 

For illustration we construct a full-rank example with a small number of subregions to check. For a given value of $n$, we generate $n$ random, linearly independent vectors $(q_i)_{i = 1}^n$ with values in the unit interval $[0,1]$. We construct the matrix $Q$ as 
\[
Q = \sum_{i = 1}^n q_i q_i^T.
\]
Note that since the $q_i$ are constructed to be linearly independent, the matrix $Q$ will be invertible. Using this decomposition of $Q$ and the constraints as before, there will be only one non-empty subregion of the feasible region. This is because the entries of each $q_i$ are positive, and hence $\ang{q_i,x} \geq 0$ for all feasible $x$. All other matrices are constructed in the same manner as before. 

\begin{table*}
\centering
\ra{1.2}
\setlength\tabcolsep{2mm}
\begin{tabular}{@{}rrrr|r@{}}\toprule
& & \multicolumn{2}{r}{CPU Time (sec)} &   \multicolumn{1}{c}{Alg 4 Error}  \\ \cmidrule{3-5}
& & Alg 1 &  Alg 4  & Error ($\%$)   \\ \midrule
$n = 20$ & $T = 1$ & 0.2189 & 0.01942 & 0 \\ 
        ~ & 10& 0.6094 & 0.02111 & 0 \\ 
        ~  & 30& 0.9902 & 0.02257 & 0.02936 \\ 
        ~ &50& 1.766 & 0.04667 & 0 \\ 
        ~ & ~ & ~ & ~ \\ 
        35 & 1& 1.423 & 0.02659 & 0.08328 \\ 
        ~ &10& 6.945 & 0.02921 & 0.01075 \\ 
        ~ &30& 7.726 & 0.03527 & 0.05105 \\ 
        ~ &50& 20.61 & 0.05936 & 0.02039 \\ 
        ~ & ~ & ~ & ~ \\ 
        50 &1& 6.25 & 0.03352 & 0 \\ 
        ~ &10& 6.753 & 0.03664 & 0.02419 \\ 
        ~ &30& 25.31 & 0.04325 & 0.01203 \\ 
        ~ &50& 189 & 0.09199 & 0.01428 \\ 
 \bottomrule
\end{tabular}
\caption{CPU Time (sec) and Error in Full Rank Example}
\label{fullrank}
\end{table*}

It is shown in Table \ref{fullrank} that the accuracy of Algorithm \ref{our_algo2} remains quite high using this decomposition of $Q$. So, the accuracy of Algorithm \ref{our_algo2} should not be attributed primarily to the number of subregions. Therefore, since choosing a decomposition of $Q$ which results in few nonempty subregions yields a faster algorithm, it is advantageous to choose one which yields the fewest number of nonempty subregions of the feasible region. 

\section{Conclusion}\label{conc} 
We have presented an efficient and accurate algorithm for globally maximizing low-rank convex-convex quadratic fractional programming problems. We have also demonstrated that this algorithm can be utilized in high-rank problems if the numerator admits a decomposition which divides the feasible region into a small number of subregions. Although Algorithm \ref{our_algo2} is only guaranteed to converge to a local maximum of (\ref{original_problem}), we have shown heuristically that it almost always converges to the global solution. 

To guarantee global convergence of Algorithm \ref{our_algo2}, one needs only to have a method to determine if there is a feasible $x \in X\cap R$ such that $F(x) > F(x^*)$ where $x^*$ is the local maximum found in the subregion $R$ by Algorithm \ref{quad_trans_algo}. One way this may be done is by using a solver to maximize the nonconvex quadratic
\[
G(x) = x^T(Q - F(x^*)P)x
\]
with the added quadratic constraint that $G(x) > 0$, and artificially terminating the solver once a feasible $x$ is found. If no such $x$ can be found, then $x^*$ is the global solution in the subregion. 

Finally, we remark that Algorithm \ref{our_algo2} is also applicable to the quadratic sum-of-ratios problems, i.e. in the case when the matrix  $P$ in the denominator of \eqref{rearrange} is allowed to vary with $m$ (in fact, the denominators need not be quadratic, they just need to be convex).

\bigskip

\noindent {\bf{Acknowledgement}.} Both authors of the paper were supported in part by the NSF grant DMS-2208031. The paper is dedicated to the everlasting memory of Guido L.~Weiss whose research, teaching, and friendship has inspired generations.

\bibliographystyle{plain}
\bibliography{bibliography}

\begin{thebibliography}{10}

\bibitem{benson_2002}
H.~P. Benson.
\newblock Global optimization algorithm for the nonlinear sum of ratios
  problem.
\newblock {\em Journal of Optimization Theory and Applications}, 112(1):1–29,
  2002.

\bibitem{benson1_2002}
Harold~P. Benson.
\newblock Using concave envelopes to globally solve the nonlinear sum of ratios
  problem.
\newblock {\em Journal of Global Optimization}, 22(1/4):343–364, Jan 2002.

\bibitem{charnes_cooper_1962}
A.~Charnes and W.~W. Cooper.
\newblock Programming with linear fractional functionals.
\newblock {\em Naval Research Logistics Quarterly}, 10(1):273–274, 1962.

\bibitem{dinkelbach_1967}
Werner Dinkelbach.
\newblock On nonlinear fractional programming.
\newblock {\em Management Science}, 13(7):492–498, 1967.

\bibitem{gao_mishra_shi_2010}
Lianbo Gao, Shashi~K. Mishra, and Jianming Shi.
\newblock An extension of branch-and-bound algorithm for solving
  sum-of-nonlinear-ratios problem.
\newblock {\em Optimization Letters}, 6(2):221–230, 2010.

\bibitem{goodearl_lenagan_2012}
K.R. Goodearl and T.H. Lenagan.
\newblock Lu decomposition of totally nonnegative matrices.
\newblock {\em Linear Algebra and its Applications}, 436(7):2554–2566, 2012.

\bibitem{gotoh_konno_2001}
Jun-Ya Gotoh and Hiroshi Konno.
\newblock Maximization of the ratio of two convex quadratic functions over a
  polytope.
\newblock {\em Computational Optimization and Applications}, 20(1):43–60,
  2001.

\bibitem{HNT22}
Longxiu Huang, Deanna Needell, and Sui Tang.
\newblock Robust recovery of bandlimited graph signals via randomized dynamical
  sampling.
\newblock arXiv:2109.14079, 2021.

\bibitem{ibaraki_1983}
Toshihide Ibaraki.
\newblock Parametric approaches to fractional programs.
\newblock {\em Mathematical Programming}, 26(3):345–362, 1983.

\bibitem{jiao_liu_2017}
Hongwei Jiao and Sanyang Liu.
\newblock An efficient algorithm for quadratic sum-of-ratios fractional
  programs problem.
\newblock {\em Numerical Functional Analysis and Optimization},
  38(11):1426–1445, 2017.

\bibitem{lo_mackinlay_1997}
Andrew Lo and A.~Craig MacKinlay.
\newblock Maximizing predictability in the stock and bond markets.
\newblock {\em Macroeconomic Dynamics}, 1997.

\bibitem{qu_zhang_zhao_2007}
Shao-Jian Qu, Ke-Cun Zhang, and Jia-Kun Zhao.
\newblock An efficient algorithm for globally minimizing sum of quadratic
  ratios problem with nonconvex quadratic constraints.
\newblock {\em Applied Mathematics and Computation}, 189(2):1624–1636, 2007.

\bibitem{shen_yu_2018}
Kaiming Shen and Wei Yu.
\newblock Fractional programming for communication systems—part i: Power
  control and beamforming.
\newblock {\em IEEE Transactions on Signal Processing}, 66(10):2616–2630,
  2018.

\bibitem{yamamoto_konno_2007}
R.~Yamamoto and H.~Konno.
\newblock An efficient algorithm for solving convex–convex quadratic
  fractional programs.
\newblock {\em Journal of Optimization Theory and Applications},
  133(2):241–255, 2007.

\end{thebibliography}

\end{document}